\documentclass[11pt,a4paper,twoside,bibliography=totoc]{scrartcl}
\usepackage{a4wide}
\usepackage{amsfonts}
\usepackage{amsmath}
\usepackage{amssymb}
\usepackage{array}
\usepackage{amsthm}
\usepackage{subfigure}
\usepackage{subfig}
\usepackage{mathdots}

\usepackage{algorithm}%,algorithmic}
\usepackage{algorithmic}

\usepackage{graphicx}
\usepackage{color}
\usepackage{pgf}%,pgfarrows,pgfnodes} % pgfarrows, pgfnodes obsolete
\usepackage{scrpage2}
\usepackage{multirow}
\usepackage{listings} %Quellcode einbinden
\usepackage{tikz}
\usepackage{graphicx}
		
\newcommand{\N}{\ensuremath{\mathbb{N}}}
\newcommand{\T}{\ensuremath{\mathbb{T}}}

\newcommand{\NZ}{\ensuremath{\mathbb{N}_{0}}}
\newcommand{\Z}{\ensuremath{\mathbb{Z}}}

\newcommand{\R}{\ensuremath{\mathbb{R}}}
\newcommand{\C}{\ensuremath{\mathbb{C}}}

\newcommand{\ii}{\textnormal{i}}
\newcommand{\e}{\textnormal{e}}
\newcommand{\supp}{\textnormal{supp}}
\newcommand{\eip}[1]{\textnormal{e}^{2\pi\ii{#1}}}
\newcommand{\eim}[1]{\textnormal{e}^{-2\pi\ii{#1}}}
\newcommand{\norm}[1]{\left\Vert #1\right\Vert}

\newcommand{\set}[1]{\left\{ #1 \right\}}

\DeclareMathOperator*{\diag}{diag}

\newtheorem{thm}{Theorem}[section]
\newtheorem{lemma}[thm]{Lemma}
\newtheorem{remark}[thm]{Remark}
\newtheorem{definition}[thm]{Definition}
\newtheorem{example}[thm]{Example}
\newtheorem{corollary}[thm]{Corollary}
\newtheorem{proposition}[thm]{Proposition}

\numberwithin{equation}{section}
\numberwithin{table}{section}
\numberwithin{figure}{section}

\newcommand{\bend}{\hspace*{0ex} \hfill \hbox{\vrule height
    1.5ex\vbox{\hrule width 1.4ex \vskip 1.4ex\hrule  width 1.4ex}\vrule
    height 1.5ex}}

\long\def\symbolfootnote[#1]#2{\begingroup%
\def\thefootnote{\fnsymbol{footnote}}\footnote[#1]{#2}\endgroup}
\newcommand{\dd}{\mathrm{d}}

\renewcommand{\mathbf}[1]{\ensuremath{\boldsymbol{#1}}}

\newcommand{\rank}{ \operatorname{rank}}

\newcommand{\im}{\operatorname{im}}

\renewcommand{\thefootnote}{\fnsymbol{footnote}}

\allowdisplaybreaks
\title{Prony's method under an almost sharp multivariate Ingham inequality}

\date{\today}
\date{}

\author{
Stefan Kunis\footnotemark[2]
\and H.~Michael M\"oller\footnotemark[3]
\and Thomas Peter\footnotemark[2]
\and Ulrich von der Ohe\footnotemark[2]}

\begin{document}

\maketitle

\begin{abstract}
The parameter reconstruction problem in a sum of Dirac measures from its low frequency trigonometric moments is well understood in the univariate case and has a sharp transition of identifiability with
respect to the ratio of the separation distance of the parameters and the order of moments.
Towards a similar statement in the multivariate case, we present an Ingham inequality which improves the previously best known dimension-dependent constant from square-root growth to a logarithmic one. 
Secondly, we refine an argument that an Ingham inequality implies identifiability in multivariate Prony methods to the case of commonly used max-degree by a short linear algebra argument, closely related to
a flat extension principle and the stagnation of a generalized Hilbert function.

\medskip

\medskip

\end{abstract}

\footnotetext[2]{
  Osnabr\"uck University, Institute of Mathematics
  \texttt{\{skunis,petert,uvonderohe\}@uos.de}
}

\footnotetext[3]{
  TU Dortmund, Fakult\"at f\"ur Mathematik
  \texttt{moeller@mathematik.tu-dortmund.de}
}

%%%%%%%%%%%%%%%%%%%%%%%%%%%%%%%%%%%%%%%%%%%%%%%%%%%%%%%%%%%%%%%%%%%%%%%%%%%%%%
\section{Introduction}
\label{sect:intro}
 The reconstruction of a Dirac ensemble from its low frequency trigonometric moments or equivalently the parameter reconstruction problem in exponential sums has been studied since \cite{Pr95} as Prony's method,
 see e.g.~\cite{Sc86,RoKa90,HuSa90,VeMaBl02,PoTa10,FiMhPr12} and the survey \cite{PlTa14}.
 Popular different approaches to this reconstruction problem include, without being exhaustive, convex optimization \cite{CaFe14,CaFe13,BeDeFe15a,BeDeFe15b}, where the problem has been termed `super-resolution' and can be solved via a lifting technique by semidefinite optimization, and maximum likelihood techniques \cite{TuKu82,BrMa86,Ca88,ClSc92,RaAr92,DM94,NaKa94,LeVaVHDM01}.
 Multivariate variants of Prony-like methods have been considered e.g. in~\cite{JiSiBe01,ShDr07,AnCaHo10,PoTa132,PlWi13,DiIs15,KuPeRoOh16,Sa16,KuMoOh16} using generic arguments, projections to univariate problems, or
 Gr\"obner basis techniques.
 
 The previously best known condition for an Ingham inequality to hold is $nq>\sqrt{d}/2$, where $n$ denotes the order of the used moments, $q$ the separation distance of the parameters, and $d$ the dimension, respectively. Together with a Gr\"obner basis construction which asks for a total-degree setting and thus introduces a factor $d$, this results in a deterministic a-priori bound $(n-d-1)q>d^{3/2}$ for the success of Prony's method, see also \cite{KuMoOh16}.
 
 In this paper we refine an appropriate Ingham inequality as well as the algebra techniques to improve this a-priori condition to $(n-1)q>3+2\log d$.
 After fixing notation, Section \ref{sec:ingham} improves the Ingham inequality \cite{In36,KoLo04,PoTa132} by constructing a function with compact support in space domain and zero crossing at an $\ell^p$-ball in frequency domain.
 In Section \ref{sec:prony}, we refine our previous analysis \cite{KuMoOh16} of the vanishing ideal of the parameter set and prove that an interpolation condition, equivalent to some full rank condition, implies that the parameters can be identified by polynomials of certain max-degree.
 In total, the multivariate Prony method correctly identifies the parameters under a deterministic condition which is sharp up to a logarithmic factor in the space dimension.

\section{Main results}\label{sect:main}
In what follows, we establish a deterministic condition for the success of Prony's method relying on two ingredients - a new Ingham inequality and an analysis of the parameters' vanishing ideal.
Ingham's inequality generalizes Parseval's identity in the sense that a norm equality between a space and a frequency representation of a vector is replaced by a certain norm equivalence.
% For simplicity we consider the univariate case $d=1$ first. Ingham's inequality generalizes Parseval's identity in the sense, that for arbitrary $\hat f=(\hat f_j)_{j=1,\hdots,M}$ the norm equality
% \begin{equation*}
%  M \|\hat f\|_2 = \left\|\sum_{j=1}^M \hat{f}_j \omega_M^{jk}\right\|_2,\qquad \omega_M=\eip{/M},
% \end{equation*}
% is replaced by the norm equivalence
% \begin{equation*}
%  c \|\hat f\|_2 \le \left\|\sum_{j=1}^M\hat{f}_j z_j^k\right\|_2 \le C \|\hat f\|_2,
% \end{equation*}
% where $z_j\in\C$, $|z_j|=1$, are fixed nodes on the complex unit circle.
With respect to the identifiability in Prony's method, we are interested in establishing conditions on the parameters such that a non trivial lower bound in this inequality exists.
Theorem \ref{thm:ingham} and Corollary \ref{cor:SepVand} give such a result under a separation condition on the parameters and improve upon \cite{In36,KoLo04,PoTa132} by a smaller space dimension dependent constant.
Subsequently, Theorem \ref{thm:kerA} and Corollary \ref{cor:prony} apply the established Ingham inequality to prove success of Prony's method by adapting standard arguments from
algebraic geometry \cite[App.~B]{Va98} to our so-called max-degree setting.

\subsection{Preliminaries}\label{sec:prelim}
Throughout the paper, $d\in\N$ always denotes the dimension and $\T^d:=\{z\in\C:|z|=1\}^d$ the torus with parameterization $\T^d\ni z=\eim{t}$ for a unique $t\in[0,1)^d$.
Now let $M\in\N$, pairwise distinct parameters $t_j\in[0,1)^d$, $j=1,\hdots,M$, be given, and define their (wrap-around) separation distance by
\begin{equation*}
 q:=\min_{r\in\Z^d,\;j\ne\ell}\|t_j-t_{\ell}+r\|_{\infty}.
\end{equation*} 
For given coefficients $\hat{f}_j\in\C\setminus\{0\}$, the trigonometric moment sequence of the complex Dirac ensemble
$\tau:\mathcal{P}([0,1)^d)\rightarrow\C$, $\tau=\sum_{j=1}^M\hat{f}_j\delta_{t_j}$, is the exponential sum
\begin{equation*}
f\colon\Z^d\to\C,
\quad
k\mapsto\int_{[0,1)^d}\eim{kt}\dd\tau(t)
=\sum_{j=1}^M\hat{f}_j z_j^k,
\end{equation*}
with parameters $z_j:=\eim{t_j}=(\eim{t_{j,1}},\hdots,\eim{t_{j,d}})\in\T^d$.
A convenient choice for the truncation of this sequence is $\|k\|_{\infty}:=\max\{|k_1|,\hdots,|k_d|\}\le n$ and our aim is to reconstruct the
parameters $t_j\in[0,1)^d$ and coefficients $\hat f_j\in\C\setminus\{0\}$ from the trigonometric moments $f(k)$, $k\in\Z^d$, $\|k\|_{\infty}\le n$.

Prony's method tries to realize the parameters $z_j$ as common zeros of certain polynomials as follows.
We identify $\C^{(n+1)^d}$ with the space of polynomials of max-degree $n$, i.e., $\Pi_n:=\{p:p(z)=\sum_{k\in\NZ^d,\|k\|_{\infty}\le n} p_k z^k\}$, and for $K\subset\C^{(n+1)^d}$ we let
\begin{equation*}
 V(K):=\{z\in\C^d:p(z)=0\text{ for all } p\in K\}.
\end{equation*}
Given the above moments, it turns out that an appropriate set of polynomials is the kernel of the multilevel Toeplitz matrix
\begin{equation*}
  T:=T_n:=\left(f(k-\ell)\right)_{k,\ell\in\NZ^d,|k|,|\ell|\le n}\in\C^{(n+1)^d\times(n+1)^d}.
 \end{equation*}
 Direct computation easily shows the factorization $T=A^* D A$ with diagonal matrix $D:=\diag(\hat f_j)\in\C^{M\times M}$ and Vandermonde matrix
\begin{equation*}
  A:=A_n:=\left(z_j^k\right)_{\substack{j=1,\dots,M\\k\in\NZ^d,\|k\|_{\infty}\le n}}\in\C^{M\times(n+1)^d},\qquad z_j:=\eim{t_j}\in\C^d,
\end{equation*}
which is our entry point for the subsequent analysis of the problem.

\subsection{Ingham's inequality}\label{sec:ingham}
The univariate case is well understood and the popular paper \cite{Mo15} used a so-called Beurling-Selberg majorant and minorant to prove a discrete Ingham inequality (which has been generalized to
the unit disk recently \cite{AuBo17}).
Tensorizing the majorant easily gives an upper bound in a multivariate discrete Ingham inequality but simple attempts to provide a minorant and thus a lower bound failed, see e.g.~\cite{Li15}.
A construction of a valid minorant by linear combinations of univariate majorants and minorants can be found e.g.~in \cite{CaGoKe16}, these seem to become effective only if $nq>C\cdot d$, i.e.,
show a linear dependence in the space dimension.
A similar bound follows from \cite{KuPo07} using localized trigonometric polynomials and a packing argument.
Moreover note that the Fourier analytic approach to sphere packing problems \cite{CoEl03,Vi16} asks for a similar but not quite the same construction of a function.
The classical and currently best known construction \cite{In36,KoLo04,PoTa132} uses an eigenfunction of the Laplace operator on the cube and becomes effective if $nq>\frac{1}{2}\sqrt{d}$.
Subsequently, we replace the Laplace operator by the sum of higher order derivatives, improve the previous bound for $d>5$, and show a logarithmic growth of the space dependent constant.

\begin{lemma}\label{lem:PropPhi}
Let $d,r\in\N$, $p:=2r$, $n,q>0$, and the functions $\varphi:\R\rightarrow\R$,
\begin{equation*}
 \varphi(x):=\begin{cases}
           \left(1-\left(\frac{2x}{q}\right)^2\right)^{r},& |x|<\frac{q}{2},\\
           0, & \text{otherwise},
          \end{cases}
\end{equation*}
and $\psi:\R^d\rightarrow\R$,
\begin{equation*}
  \psi
  =\left(\left(2\pi n\right)^p-(-1)^r\sum_{s=1}^d \frac{\partial^{p}}{\partial x_s^{p}}\right)\bigotimes_{\ell=1}^d \varphi*\varphi,
 \end{equation*}
be given, see Figures \ref{fig:p2} and \ref{fig:p4} for examples, then
\begin{enumerate}
 \item the Fourier transform $\hat\psi(v):=\int_{\R^d}\psi(x)\eim{vx}\dd x$ is bounded and obeys
 \begin{equation*}
   \hat\psi(v)\begin{cases}
	\ge 0, & \|v\|_{p}\le n\\
	\le 0, & \|v\|_{p}\ge n,
              \end{cases}
 \end{equation*}
 \item $\supp\psi=[-q,q]^d$,
 \item and $\psi(0)>0$ if $nq>C_p \sqrt[p]{d}$ with $C_p\le (2p+3)/(\e\pi)$.
\end{enumerate}
\end{lemma}

\begin{proof}
 We have the weak $r$-th derivative
 \begin{equation*}
 \varphi^{(r)}(x)=
  \begin{cases}
           (-1)^r q^{-r} 4^r r! P_r\left(\frac{2 x}{q}\right),& |x|<\frac{q}{2},\\
           0, & \text{otherwise},
          \end{cases} 
 \end{equation*}
 where $P_r$ denotes the $r$-th Legendre polynomial with normalization $P_r(1)=1$.
 This derivative is of bounded variation and thus, the Fourier transform of the function $\varphi$ obeys $|\hat\varphi(v)|\le C(1+|v|)^{-r-1}$.
 Now the first assertion follows from
 \begin{equation*}
  \hat \psi(v)
  =\left(\left(2\pi n\right)^p-\sum_{s=1}^d \left(2\pi v_s\right)^p\right) \prod_{\ell=1}^d \left(\hat\varphi(v_\ell)\right)^2.
 \end{equation*}
 The second claim easily follows from $\supp\varphi=[-\frac{q}{2},\frac{q}{2}]$.
 Moreover, we have
 \begin{equation*}
  \varphi*\varphi(0)=\frac{q}{2}\int_{-1}^1 \left(1-x^2\right)^p \dd x = \frac{q\sqrt{\pi} p!}{2\Gamma(p+\frac{3}{2})}
 \end{equation*}
 and noting $\varphi^{(r)}$ being odd for $r$ odd, we get
 \begin{equation*}
 \varphi^{(r)}*\varphi^{(r)}(0)
  =\left(\frac{4^r r!}{q^r}\right)^2 \frac{(-1)^r q}{2} \int_{-1}^1 (P_r(x))^2 \dd x
  =\frac{4^p (r!)^2 (-1)^r}{(p+1)q^{p-1}}.
 \end{equation*}
 Finally, this implies
 \begin{align*}
  \psi(0)
  &=\left(\varphi*\varphi(0)\right)^{d-1}\left(\left(2\pi n\right)^p \varphi*\varphi(0) - (-1)^r d \cdot \varphi^{(r)}*\varphi^{(r)}(0)\right)\\
  &=\left(\varphi*\varphi(0)\right)^{d-1}\left(\left(2\pi n\right)^p \frac{q\sqrt{\pi} p!}{2\Gamma(p+\frac{3}{2})} - \frac{d 4^p (r!)^2}{(p+1)q^{p-1}}\right)
  >0,
 \end{align*} 
 provided the term in brackets is positive. Using the Legendre duplication formula for $\Gamma(p+2)$, this is equivalent to $nq>C_p \sqrt[p]{d}$ with a constant
 \begin{equation}\label{eq:Cp}
  C_p:=\left(\frac{\Gamma\left(\frac{p}{2}+1\right)\Gamma\left(p+\frac{3}{2}\right)}{\pi^p\Gamma\left(\frac{p+3}{2}\right)}\right)^{\frac{1}{p}}
  \le \frac{\left(\Gamma\left(p+\frac{3}{2}\right)\right)^{\frac{1}{p}}}{\pi} \le \frac{2p+3}{\e\pi},
 \end{equation}
 where the first inequality follows from monotonicity and the second by Stirling's approximation.
\end{proof}

\begin{remark}\label{rem:p2}
 The construction of Lemma \ref{lem:PropPhi} for $p=2$ reads as
 \begin{equation*}
  \varphi(x):=\begin{cases}
           1-\left(\frac{2x}{q}\right)^2,& |x|<\frac{q}{2},\\
           0, & \text{otherwise},
          \end{cases}
 \end{equation*}
 leading to 
 \begin{equation*}
  \varphi*\varphi(0)=\frac{q}{2}\int_{-1}^{1} (1-x^2)^2\dd x=\frac{8q}{15},\qquad
  \varphi'*\varphi'(0)=-\frac{4q}{3}\cdot \frac{4}{q^2},
 \end{equation*}
 and we get $\psi(0)>0$ for $nq>\frac{\sqrt{5/2}}{\pi} \cdot \sqrt{d}\approx 0.5033 \sqrt{d}$.
 
 A minor improvement is possible by choosing the function $\varphi:\R\rightarrow\R$,
\begin{equation*}
 \varphi(x):=\begin{cases}
           \cos\frac{\pi x}{q},& |x|<\frac{q}{2},\\
           0, & \text{otherwise},
          \end{cases}
\end{equation*}
which has a %weak $1$-st derivative and a 
distributional $2$-nd derivative $\varphi''=-\pi^2 q^{-2}\varphi+C\delta_{\pm \frac{q}{2}}$ and thus
\begin{equation*}
 \varphi'*\varphi'(0)=\varphi*\varphi''(0)=-\frac{\pi^2}{q^2}\varphi*\varphi(0).
\end{equation*}
 Proceeding as in Lemma \ref{lem:PropPhi}, we get $\psi(0)>0$ for $nq>\frac{1}{2} \cdot \sqrt{d}$, see also \cite{KoLo04,PoTa132}.
\end{remark}

\begin{figure}[htbp]
\centering
\subfigure[Contour plot of $\psi$, compactly supported in the dashed square, positive at origin.]
  {\includegraphics[width=0.45\textwidth]{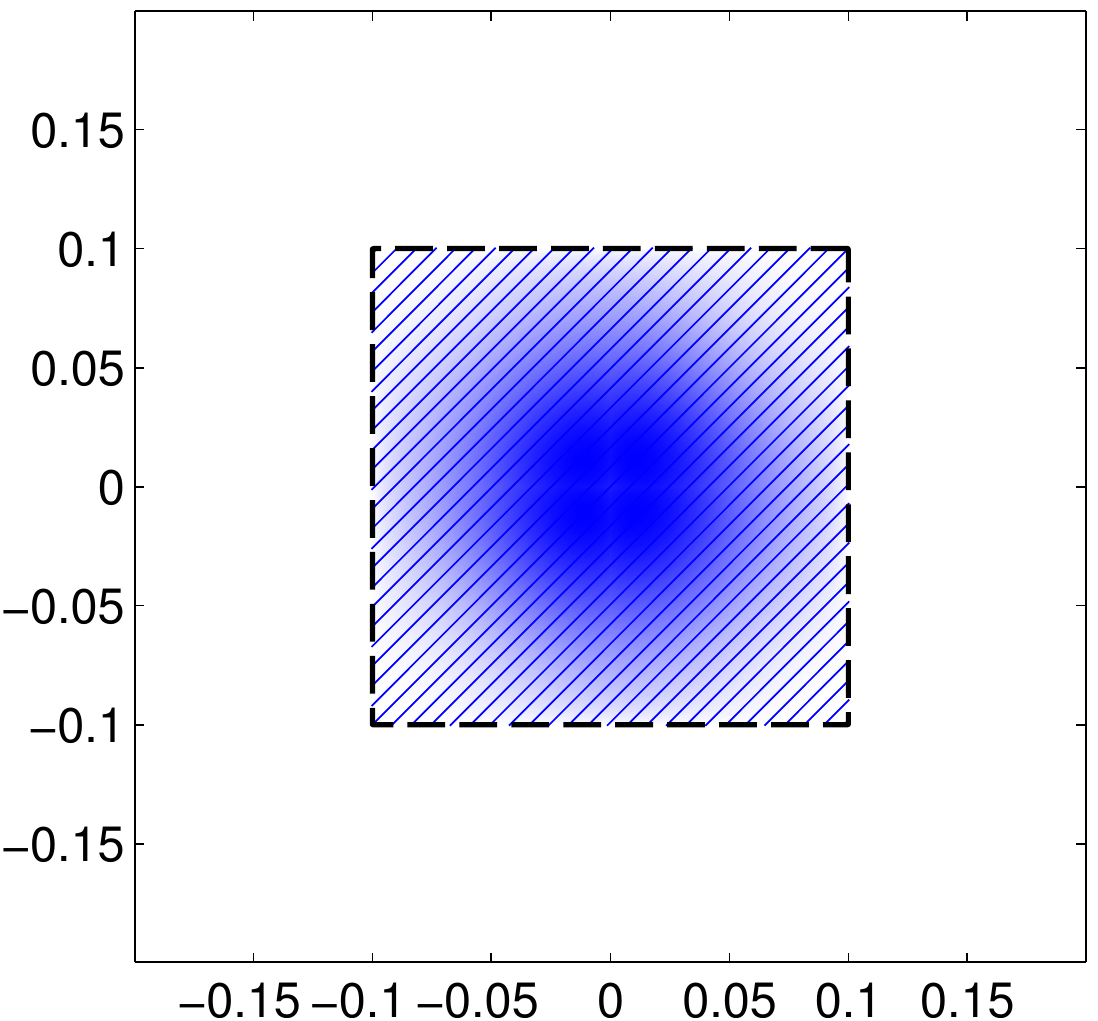}}
  \subfigure[Contour plot of $\hat \psi$, positive inside and non-positive outside the dashed $\ell^2$-ball.]
  {\includegraphics[width=0.45\textwidth]{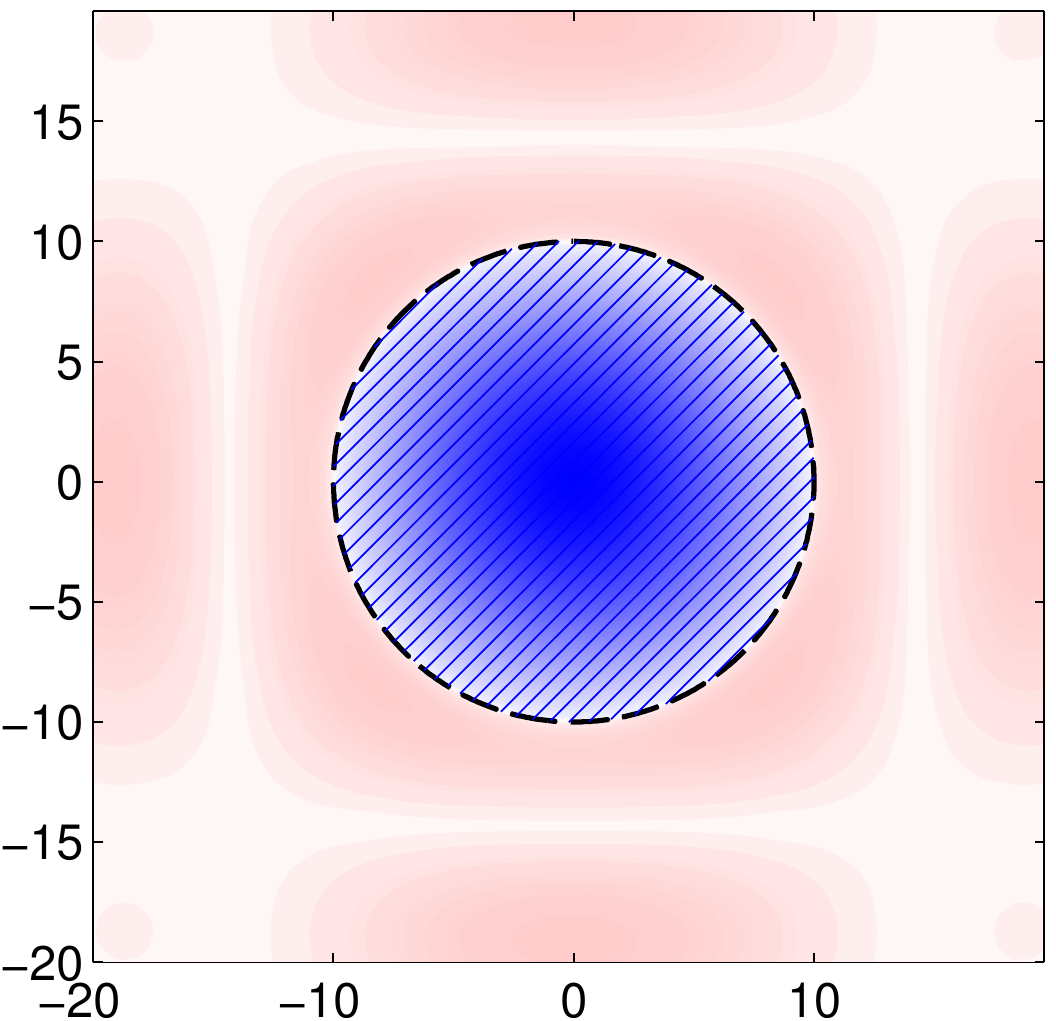}}
\caption{$p=2$, $d=2$, $q=0.1$, $n=10$, positive values blue hatched, negative values red.}
\label{fig:p2}
\end{figure}

\begin{remark}\label{rem:p4}
The construction of Lemma \ref{lem:PropPhi} for $p=4$ reads as
 \begin{equation*}
  \varphi(x):=\begin{cases}
           \left(1-\left(\frac{2x}{q}\right)^2\right)^2,& |x|<\frac{q}{2},\\
           0, & \text{otherwise},
          \end{cases}
 \end{equation*}
 leading to 
 \begin{equation*}
  \varphi*\varphi(0)=\frac{q}{2}\int_{-1}^{1} (1-x^2)^4\dd x=\frac{128q}{315},\qquad
  \varphi''*\varphi''(0)=\frac{64q}{5}\cdot \frac{16}{q^4},
 \end{equation*}
 and we get $\psi(0)>0$ for $nq>\frac{\sqrt[4]{63/2}}{\pi} \cdot \sqrt[4]{d}\approx 0.7541\cdot \sqrt[4]{d}$.
 
 An alternative choice is the function $\varphi:\R\rightarrow\R$,
\begin{equation*}
 \varphi(x):=\begin{cases}
           1+\cos\frac{2\pi x}{q},& |x|<\frac{q}{2},\\
           0, & \text{otherwise},
          \end{cases}
\end{equation*}
which has a weak $2$-nd derivative of bounded variation
 \begin{equation*}
 \varphi''(x)
 = \frac{4\pi^2}{q^2}\left(\chi_{\left(-\frac{q}{2},\frac{q}{2}\right)}(x)-\varphi(x)\right)
 %= \begin{cases}
 %          -\frac{4\pi^2}{q^2}\cos\frac{2\pi x}{q},& |x|<\frac{q}{2},\\
 %          0, & \text{otherwise},
 %         \end{cases} 
 \end{equation*}
 and the Fourier transforms obey the bounds
 %, explicitly given by
 %\begin{equation*}
 % \hat\varphi(v)
 % =\left(\left(\delta_0+\frac{1}{2}\delta_{-\frac{1}{q}}+\frac{1}{2}\delta_{\frac{1}{q}}\right) * \frac{\sin\pi q\cdot}{\pi\cdot}\right)(v)
 % =\frac{\sin\pi q v}{\pi v}-\frac{\sin\pi q v}{2\pi (v-\frac{1}{q})}-\frac{\sin\pi q v}{2\pi (v+\frac{1}{q})},
 %\end{equation*}
 $|\hat \varphi(v)|\le C(1+|v|)^{-3}$ and $|\hat \psi(v)|\le C'$.
 Moreover, we have
 \begin{equation*}
 \varphi''*\varphi''
  =\frac{16\pi^4}{q^4}\left(\varphi*\varphi-2\varphi*\chi_{\left(-\frac{q}{2},\frac{q}{2}\right)} + 
  \chi_{\left(-\frac{q}{2},\frac{q}{2}\right)}*\chi_{\left(-\frac{q}{2},\frac{q}{2}\right)}\right),
 \end{equation*}
 $\varphi*\varphi(0)=\frac{3}{2} q$, and $\varphi*\chi_{\left(-\frac{q}{2},\frac{q}{2}\right)}(0)=q$, which yields
 \begin{align*}
  \psi(0)
  &=\left(\varphi*\varphi(0)\right)^{d-1}\left(16\pi^4 n^4 (\varphi*\varphi(0))-d (\varphi''*\varphi''(0))\right)\\
  %&=16\pi^4 \|\varphi\|_2^{2d-2}\left(n^4 \|\varphi\|_2^2 - d q^{-4} \|\varphi\|_2^2 + 2 d q^{-4} \|\varphi\|_1 - d q^{-3}\right)\\
  &=3\cdot 8\pi^4 \left(\varphi*\varphi(0)\right)^{d-1} q \left( n^4 - \frac{d}{3} q^{-4}\right)>0
 \end{align*} 
 if $nq>\sqrt[4]{d/3}\approx 0.7598\cdot \sqrt[4]{d}$.

 Finally, a minor improvement is possible as follows.
 Let $\sigma\approx 2.365$ be the first positive root of $\cos t \sinh t + \cosh t \sin t$ and the function $\varphi:\R\rightarrow\R$,
\begin{equation*}
 \varphi(x):=\begin{cases}
           \frac{\cosh\sigma}{\cosh\sigma-\cos\sigma}\cos(2\sigma x/q) - \frac{\cos\sigma}{\cosh\sigma-\cos\sigma}\cosh(2\sigma x/q),& |x|<\frac{q}{2},\\
           0, & \text{otherwise},
          \end{cases}
\end{equation*}
be given, then $\varphi$ solves the biharmonic eigenvalue problem
\begin{equation*}
 \frac{\dd^4}{\dd x^4}\varphi(x)=\frac{16\sigma^4}{q^4}\varphi(x),\; |x|<\frac{q}{2},\qquad \varphi(\pm \frac{q}{2})=\varphi'(\pm \frac{q}{2})=0.
\end{equation*}
Globally, we have a $1$-st derivative, a weak $2$-nd derivative, and distributional $3$-rd and $4$-th derivatives.
In particular, we have $\varphi^{(4)}=16\sigma^4q^{-4}\varphi^{(4)}+C\delta_{\pm \frac{q}{2}}'$ and thus
\begin{equation*}
 \varphi''*\varphi''(0)=\varphi*\varphi^{(4)}(0)=\frac{16\sigma^4}{q^4}\varphi*\varphi(0).
\end{equation*}
 Proceeding as in Lemma \ref{lem:PropPhi}, we get $\psi(0)>0$ for $nq>\frac{\sigma}{\pi} \cdot \sqrt[4]{d}\approx 0.7528 \cdot \sqrt[4]{d}$.
\end{remark}

\begin{figure}[htbp]
\centering
\subfigure[Contour plot of $\psi$, compactly supported in the dashed square, positive at origin.]
  {\includegraphics[width=0.45\textwidth]{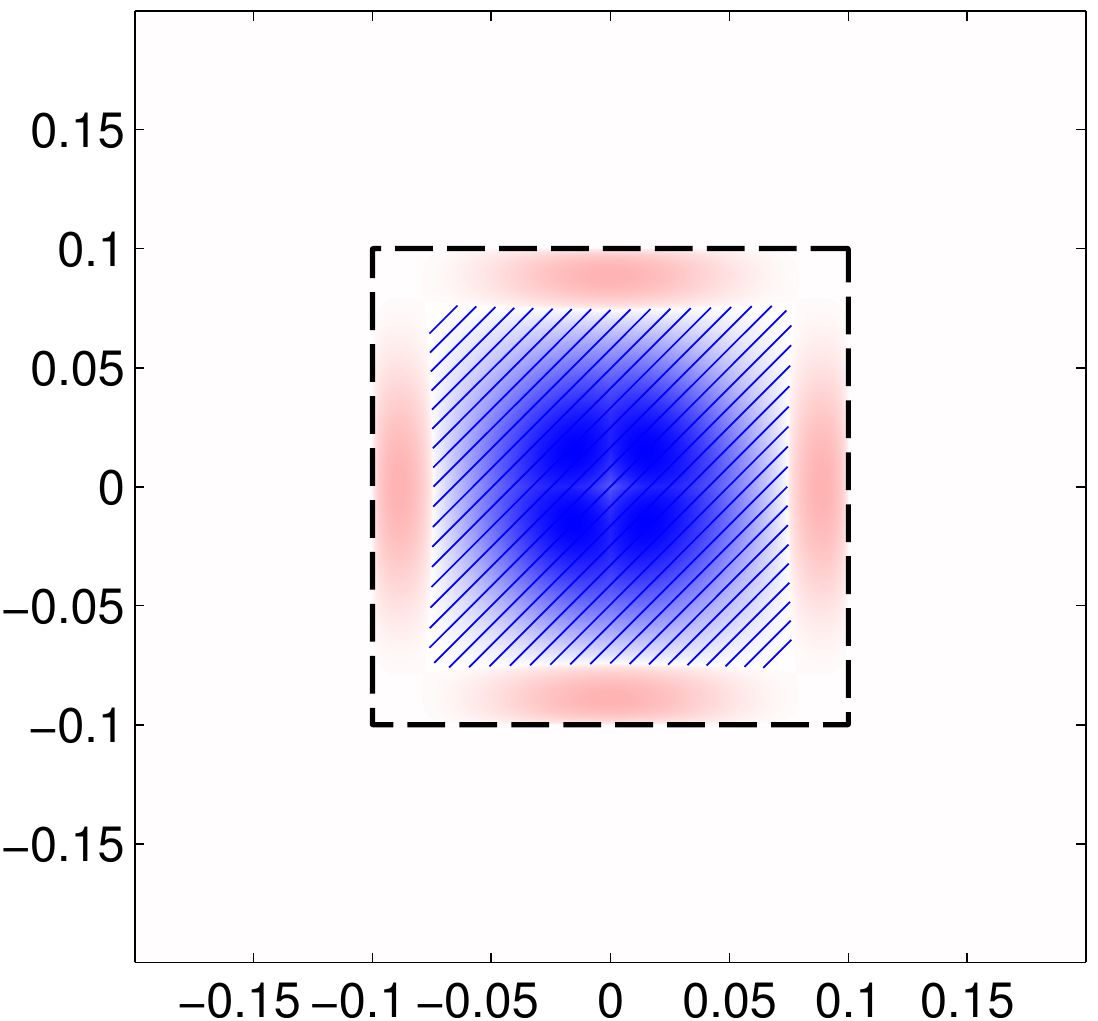}}
  \subfigure[Contour plot of $\hat \psi$, positive inside and non-positive outside the dashed $\ell^4$-ball.]
  {\includegraphics[width=0.45\textwidth]{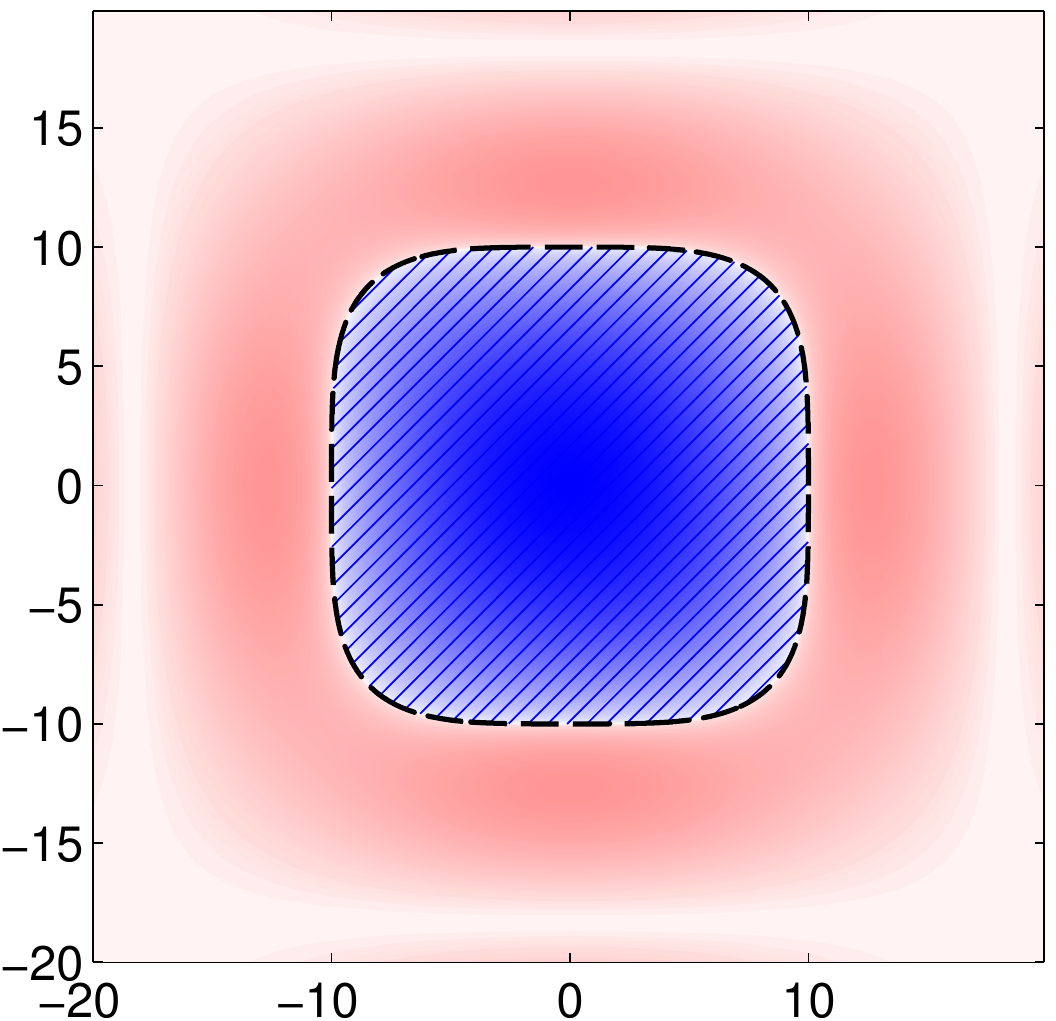}}
\caption{$p=4$, $d=2$, $q=0.1$, $n=10$, positive values blue hatched, negative values red.}
\label{fig:p4}
\end{figure}

\begin{thm}\label{thm:ingham}
 Let $n>0$, $d,p,M\in\N$, $p$ even, $t_j\in[0,1)^d$, $j=1,\hdots,M$, with
 \begin{equation*}
  \min_{r\in\Z^d,\;j\ne\ell}\|t_j-t_{\ell}+r\|_{\infty}>q
 \end{equation*}
 and $nq>\frac{2p+3}{\e\pi} \sqrt[p]{d}$, then there exists a constant $c>0$, depending only on $d,p,n,q$ such that
 \begin{equation*}
  \sum_{\substack{k\in\Z^d\\ \|k\|_{p}\le n}} \left|\sum_{j=1}^M \hat f_j \eip{k t_j}\right|^2
  \ge c \sum_{j=1}^M \left|\hat f_j\right|^2,
 \end{equation*}
 for all choices $\hat f_j\in\C$, $j=1,\hdots,M$.
\end{thm}
\begin{proof}
 Let the function $\psi$ be as in Lemma \ref{lem:PropPhi}, then
 \begin{align*}
  \max_{v\in\R^d} \hat\psi(v) \sum_{\substack{k\in\Z^d\\ \|k\|_{p}\le n}} \left|\sum_{j=1}^M \hat f_j \eip{k t_j}\right|^2
  &\ge
  \sum_{k\in\Z^d} \hat \psi(k) \left|\sum_{j=1}^M \hat f_j \eip{k t_j}\right|^2\\
  &=
  \sum_{j=1}^M \sum_{\ell=1}^M \hat f_j \overline{\hat f_\ell} \sum_{r\in\Z^d} \psi(t_j-t_{\ell} + r)\\
  &=\psi(0) \sum_{j=1}^M \left|\hat f_j\right|^2,
 \end{align*}
 where the inequality is implied by Lemma \ref{lem:PropPhi}~i), Poisson's summation formula yields the first equality, and Lemma \ref{lem:PropPhi}~ii) and the separation of the nodes $t_j$ the last equality.
 Finally, the assertion follows since Lemma \ref{lem:PropPhi}~iii) assures $\psi(0)>0$.
\end{proof}

\begin{corollary}\label{cor:SepVand}
 With the notation of Section \ref{sec:prelim}, the condition $nq>c_d$, see Table \ref{tab:cd}, or $nq>3+2\log d$ implies $\rank A_n=M$.
\end{corollary}
\begin{proof}
 First note that
 \begin{equation*}
  A_n=\left(z_j^k\right)_{\substack{j=1,\dots,M\\k\in\NZ^d,\|k\|_{\infty}\le n}}
     =\diag\left(z_j^{\left(\lceil\frac{n}{2}\rceil,\hdots,\lceil\frac{n}{2}\rceil\right)}\right) \cdot
      \left(z_j^k\right)_{\substack{j=1,\dots,M\\k\in\{-\lceil\frac{n}{2}\rceil,\hdots,\lfloor\frac{n}{2}\rfloor\}^d}}
 \end{equation*}
 and thus $A_n$ has rank $M$ if and only if the second matrix on the right hand side has this rank.
 Dropping the last column for odd $n$, this is assured by Theorem \ref{thm:ingham} if $nq>2\frac{2p+3}{\e\pi}\sqrt[p]{d}$, since
 \begin{equation*}
  \sum_{\substack{k\in\Z^d\\ \|k\|_{\infty}\le \frac{n}{2}}} \left|\sum_{j=1}^M \hat f_j \eip{k t_j}\right|^2
  \ge \sum_{\substack{k\in\Z^d\\ \|k\|_{p}\le \frac{n}{2}}} \left|\sum_{j=1}^M \hat f_j \eip{k t_j}\right|^2
  > 0
 \end{equation*}
 for $\hat f\in\C^M\setminus\{0\}$.
 Choosing $p:=2\lceil\log d\rceil$ yields the assertion since
 \begin{equation*}
  \frac{2p+3}{\e\pi}\sqrt[p]{d}\le \frac{2p+3}{\e\pi}\sqrt[p]{\e^{\frac{p}{2}}}=\frac{2p+3}{\pi\sqrt{\e}}\le \frac{7+4\log d}{\pi\sqrt{\e}}.
 \end{equation*}
\end{proof}

\begin{table}
\begin{center}
 \begin{tabular}{|r|cccccccccc|}
 \hline
 $d$ & $1$ & $2$ & $3$ & $4$ & $10$ & $16$ & $20$ & $64$ & $100$ & $256$\\
 \hline
 $c_d$ & $1.0$ & $1.4$ & $1.7$ & $2.0$ & $2.7$ & $3.0$ & $3.2$ & $4.0$ & $4.3$ & $5.0$\\
 \hline
 \end{tabular}
 \end{center}
 \caption{Explicit constants $c_d=2\min_{p\in2\N} C_p \sqrt[p]{d}$, see Equation \eqref{eq:Cp}, or its minor improvements via Remark \ref{rem:p2} and \ref{rem:p4}.\label{tab:cd}}
\end{table}

\begin{remark}
 Corollary \ref{cor:SepVand} can also be found in \cite[Cor.~4.7]{KuPo07} and \cite[Lem.~3.1]{PoTa132} under the stronger conditions $(n+1)q>2d$ and $nq>\sqrt{d}$, respectively.
 %To the best of our knowledge, it is not clear if the logarithmic growth of the $d$-dependent constant in Corollary \ref{cor:SepVand} has some reason or is just an artifact of our proof.
 Regarding a trivial lower bound of the constant $c_d$, a $d$-fold Cartesian product of equispaced parameters in each dimension consists of $M=q^{-d}$ parameters in total and thus
 $\rank A\le (n+1)^d<M$ if $(n+1)q<1$.
 Moreover, a non-trivial lower bound in \cite[Section 5.7]{OlUl12} indicates that functions as in Lemma \ref{lem:PropPhi} necessarily imply that the constant $c_d$ tends to infinity for $d\rightarrow\infty$.
 Going beyond the rank of the Vandermonde matrix $A_n$ and considering its condition number or equivalently upper and lower bounds in the Ingham inequality, \cite[Cor.~4.7]{KuPo07} implies a uniform bound
 \begin{equation*}
  \kappa(A_n W A_n^*)\le 1+\frac{2(2d)^{d+1}}{((n+1)q)^{d+1}-(2d)^{d+1}}\le 2
 \end{equation*}
 for each $q$-separated parameter set with $(n+1)q\ge 4d$ and a certain diagonal preconditioner $W$.
 On the other hand, every weaker condition $(n+1)q\le c\cdot d^{1-\varepsilon}$, $c>0$, $\varepsilon\in(0,1)$, implies at least a subexponential growing condition number, since \cite[Cor.~4.11]{KuPo07} leads to
 \begin{equation*}
  \kappa(A_n A_n^*)=\left(\frac{\lceil(n+1)q\rceil}{\lfloor(n+1)q\rfloor}\right)^d\ge\left(1+\frac{1}{(n+1)q}\right)^d
  %=\exp\left(d\log\left(1+\frac{1}{c\cdot d^{1-\varepsilon}}\right)\right)
  \ge\exp\left(d^{\varepsilon}\cdot c^{-1}\cdot\log 2\right)
 \end{equation*}
 for equispaced nodes with $(n+1)q\not\in\N$. For a numerical example, let $d=256$ and consider equispaced nodes with $\N\not\ni (n+1)q\approx 4.98 \approx c_{256}$ - while Corollary \ref{cor:SepVand} assures full rank of $A_n$, we have $\kappa(A_n A_n^*)=(5/4)^{256} \ge 10^{24}$.
\end{remark}

\subsection{Prony's method}\label{sec:prony}

As outlined in Section \ref{sec:prelim}, Prony's method tries to realize the unknown parameters $z_j$ as common roots of $d$-variate polynomials belonging to the
kernel of the multilevel Toeplitz matrix $T_n$.
In \cite{KuPeRoOh16}, we provided the well-known a-priori condition $n>M$ which however implies the need of $\mathcal{O}(M^d)$ trigonometric moments for the reconstruction of $M$ parameters.
We improved this in \cite{KuMoOh16} to the a-priori condition $(n-d-1)q>d^{3/2}$ using Gr\"obner basis arguments and a variant of the flat extension principle \cite{CuFi00,LaMo09}.
This allows for the reconstruction of $M$ well distributed parameters from $\mathcal{O}(M)$ trigonometric moments but the constant deteriorates with larger space dimension.
Subsequently, we refine a variant of the flat extension principle to our max-degree setting and give a simple linear algebra proof.
Together with the improved Ingham inequality, Prony's method succeeds under the weakened condition $(n-1)q>3+2\log d$.

\begin{lemma}\label{lem:rankA_n-stabilization}
With the notation of Section \ref{sec:prelim}, if $\rank A_n=\rank A_{n+1}$ for some $n\in\NZ$, then $\rank A_\ell=M$ for all $\ell\ge n$.
\end{lemma}
\begin{proof}
Let $r:=\rank A_n$ and pick $k_1,\dots,k_r\in\NZ^d$ with $\norm{k_i}_\infty\le n$ such that the matrix
\begin{equation*}
B:=(z_j^{k_i})_{\substack{j=1,\hdots,M\\i=1,\hdots,r}}\in\C^{M\times r} 
\end{equation*}
has rank~$r$.
Subsequently, we let $m\in\NZ^d$, $m\notin\set{k_1,\dots,k_r}$, $\|m\|_{\infty}\le n+2$ and
\begin{equation*}
B_m:=(B,(z_j^m)_{j=1,\dots,M})\in\C^{M\times r+1}\text{.}
\end{equation*}
If $\norm{m}_\infty\le n+1$, then $B_m$ is a submatrix of $A_{n+1}$ and $r=\rank B\le\rank B_m\le\rank A_{n+1}=r$ by assumption.
If $\norm{m}_\infty=n+2$, then $m=s+k$, $\|s\|_{\infty}=1$, $\|k\|_{\infty}=n+1$, implies
\begin{equation*}
 z_j^m=z_j^s\cdot z_j^k=\sum_{i=1}^r c_j z_j^{s+k_i}, \quad c_j\in\C,\;\|s+k_i\|_{\infty}\le n+1,
\end{equation*}
and thus $(z_j^m)_{j=1,\hdots,M}$ is linear dependent on the columns of $A_{n+1}$ and thus on the columns of $B$, i.e., $\rank B_m=r$.
This shows $\rank A_{n+2}=r$ and inductively $\rank A_\ell=r$ for all $\ell\ge n$
Finally, the space $\Pi_M$ of polynomials of max-degree at most~$M$ interpolates any set of $M$ points and thus $\rank A_M=M$ and $r=M$.
\end{proof}

\begin{thm}\label{thm:kerA}
 With the notation of Section \ref{sec:prelim}, if $\rank A_n=M$, then $V(\ker A_{n+1})=\{z_1,\hdots,z_M\}$.
\end{thm}
\begin{proof}
Clearly,
``$\supset$'' holds.
To prove the reverse inclusion,
assume that there is a
\begin{equation*}
z_0\in V(\ker A_{n+1})\setminus\set{z_1,\dots,z_M}\text{.}
\end{equation*}
For $\ell\in\NZ$ consider the matrix
\begin{equation*}
C_\ell:=(z_j^k)_{\substack{j=0,\dots,M\\k\in\NZ^d,\norm{k}_\infty\le\ell}}\in\C^{M+1\times(\ell+1)^d}\text{,}
\end{equation*}
i.e.~$A_\ell$ extended by the row $(z_0^k)_{\norm{k}_\infty\le\ell}$.
We show that $\ker C_\ell=\ker A_\ell$ for all $\ell\le n+1$.
The inclusion ``$\subset$'' is clear.
Let $p\in\ker A_\ell$.
Since $\ell\le n+1$,
we can consider $p$ as an element of $\ker A_{n+1}$,
which, by assumption, yields $p(z_0)=0$.
By the definition of $C_\ell$
we have $C_\ell p=0$
(considering $p$ as an element of $\ker A_\ell$ again).
Thus $\ker C_\ell=\ker A_\ell$ for all $\ell\le n+1$,
and therefore $\rank C_\ell=\rank A_\ell$.
In particular,
$\rank C_n=\rank A_n=M$
and
$\rank C_{n+1}=\rank A_{n+1}=M$,
thus $\rank C_n=\rank C_{n+1}$.
By Lemma~\ref{lem:rankA_n-stabilization}
we have $\rank C_n=\lvert\set{z_0,\dots,z_M}\rvert=M+1$,
a contradiction.
\end{proof}

\begin{remark}
Inspection of the proofs readily shows that the results of Lemma \ref{lem:rankA_n-stabilization} and Theorem \ref{thm:kerA} hold for arbitrary pairwise different $z_j\in\C^d$.
The Vandermonde matrix $A_n$ is the representation matrix of the evaluation homomorphism which maps a polynomial to its values at the prescribed points $z_j$.
In algebraic geometry and with the max-degree replaced by the total degree, the rank of such maps is called Hilbert function.
It is well known that such a function is strictly increasing with $n$ up to stagnation at the value $M$, see e.g.~\cite[App.~B]{Va98}.
The previous Lemma \ref{lem:rankA_n-stabilization} shows by elementary arguments that this is true also for the max-degree.
\end{remark}

\begin{corollary}\label{cor:prony}
 With the notation of Section \ref{sec:prelim}, the condition $(n-1)q>c_d$, see Table \ref{tab:cd}, or $(n-1)q>3+2\log d$ implies $V(\ker T_n)=\{z_1,\hdots,z_M\}$.
\end{corollary}
\begin{proof}
 Corollary \ref{cor:SepVand} implies $\rank A_{n-1}=M$ and thus $V(\ker A_n)=\{z_1,\hdots,z_M\}$ follows from Theorem \ref{thm:kerA}.
 Finally note that always $\ker A_n\subset\ker T_n$ and equality follows from $\rank A_n=M$
 since $p\in\ker T_n$ implies $A_n p\in\ker (A_n^* D)=\im(D^* A_n)^{\bot}=\{0\}$ and thus $p\in\ker A_n$.
\end{proof}

%%%%%%%%%%%%%%%%%%%%%%%%%%%%%%%%%%%%%%%%%%%%%%%%%%%%%%%%%%%%%%%%%%%%%%%%%%%%%%
\section{Summary}\label{sect:sum}

We have shown that Prony's method identifies a Dirac ensemble from its first moments provided an associated Vandermonde matrix has full rank.
This is indeed fulfilled if the number of given coefficients exceeds a constant divided by the separation distance of the unknown parameters.
Explicit bounds for the involved constant grow logarithmically in the space dimension and improve over previous constructions also for small space dimensions.

\textbf{Acknowledgment.}
The authors thank the referees for their valuable suggestions and additional pointers to related literature.
Moreover, we gratefully acknowledge support by the DFG within the research training group 1916: Combinatorial structures in geometry and by the DAAD-PRIME-program.

\bibliographystyle{abbrv}
\bibliography{../references}
\end{document}